\documentclass[npg, manuscript]{copernicus}

\usepackage{amsmath,amsfonts, amsthm,mathtools}

\usepackage{algorithm}
\usepackage{algorithmic}

\DeclareMathOperator*{\argmin}{arg\,min}
\DeclareMathOperator*{\logsumexp}{LogSumExp}

\newtheorem{theorem}{Theorem}[section]

\newtheorem{definition}{Definition}[section]
\newtheorem{remark}{Remark}[section]
\newtheorem{corollary}{Corollary}[theorem]
\newtheorem{lemma}[theorem]{Lemma}

\usepackage{tikz}
\usetikzlibrary{calc}
\usetikzlibrary{math}
\newcommand\sqdiag{\tikz{\draw (0,0) rectangle (0.7em, 0.7em); \draw (0, 0.7em) -- (0.7em, 0);}} 
\DeclareMathOperator{\bdiag}{\sqdiag}

\DeclareMathOperator{\Cov}{Cov}

\def\*#1{\boldsymbol{\mathbf{#1}}}
\def\##1{\mathfrak{#1}}

\usepackage{cleveref}
\usepackage{algorithmic}
\usepackage{algorithm}
\usepackage{array}
\usepackage[caption=false,font=normalsize,labelfont=sf,textfont=sf]{subfig}
\usepackage{textcomp}
\usepackage{stfloats}
\usepackage{url}
\usepackage{verbatim}
\usepackage{graphicx}
\usepackage{cite}

\usepackage{pgfplots}
\pgfplotsset{compat=1.18} 
\usepackage{pgfplotstable}
\usepackage{cbcolors}

\begin{document}
\nolinenumbers

\title{Ensemble-localized Kernel Density Estimation with Applications to the Ensemble Gaussian Mixture Filter}


\Author[1][andrey.a.popov@utexas.edu]{Andrey A.}{Popov} 
\Author[2]{Enrico M.}{Zucchelli}
\Author[2]{Renato}{Zanetti}

\affil[1]{Oden Institute for Computational Engineering \& Sciences, the University of Texas at Austin, Austin,
TX, 78712, USA}
\affil[2]{Dept of Aerospace Engineering \& Engineering Mechanics, the University of Texas at Austin, Austin,
TX, 78712, USA}




\runningtitle{E-localized KDE for EnGMF}

\runningauthor{A.A. Popov, E.M. Zucchelli, R. Zanetti}

\received{}
\pubdiscuss{} 
\revised{}
\accepted{}
\published{}


\firstpage{1}

\maketitle

\begin{abstract}
The ensemble Gaussian mixture filter (EnGMF) is a non-linear filter suited to data assimilation of highly non-Gaussian and non-linear models that has practical utility in the case of a small number of samples, and theoretical convergence to full Bayesian inference in the ensemble limit. We aim to increase the utility of the EnGMF by introducing an ensemble-local notion of covariance into the kernel density estimation (KDE) step for the prior distribution. We prove that in the Gaussian case, our new ensemble-localized KDE technique is exactly the same as more traditional KDE techniques. We also show an example of a non-Gaussian distribution that can fail to be approximated by canonical KDE methods, but can be approximated well by our new KDE technique.
We showcase our new KDE technique on a simple bivariate problem, showing that it has nice qualitative and quantitative properties, and significantly improves the estimate of the prior and posterior distributions for all ensemble sizes tested. We additionally show the utility of the proposed methodology for sequential filtering for the Lorenz '63 equations, achieving a significant reduction in error, and less conservative behavior in the uncertainty estimate with respect to traditional techniques.
\end{abstract}


\introduction  

The sample covariance is a good measure of the global relationship between state variables, but fails to capture local relationships between variables for distributions whose curvature is highly non-Gaussian. Canonical kernel density estimation techniques~\citep{silverman2018density,chen2017tutorial} often rely only on the global curvature of the data, and are thus ill-suited for advanced state estimation techniques relying thereon.

The first law of geography~\citep{tobler1970computer} states that ``everything is related to everything else, but near things are more related than distant things.'' In the context of state estimation and data assimilation for geophysical systems~\citep{asch2016data, reich2015probabilistic} this law begets heuristics known as covariance localization.
Common approaches to covariance localization rely on spatial distances between state and observation variables~\citep{asch2016data}. The two predominant approaches are B-localization, where the prior covariance is scaled in a way such that far away state variables do not influence each other, and R-localization, where the observation covariances are scaled in a way such that state variables that are far away from an observation are not influenced by it. These approaches fundamentally rely on a spatial understanding of state variables or observations, and are not fit for general problems where the state variables might not have a complete spatial structure. A new more general approach to localization is therefore required.

The aim of this work is to build a new kernel density estimation (KDE) technique for particle filtering.
While several computationally-intensive methods exist for density estimation in particle filters, such as clustering~\citep{yun2021clustering} and `blob' filtering~\citep{psiaki2015gaussian,psiaki2016gaussian}, our goal is to construct a method that can both perform accurate kernel density estimation while simultaneously being fast enough to utilize in an online setting.

The previous conference paper~\citep{popov2023elengmf} on which this work is based, applied a largely heuristic idea of finding a local measure of covariance through the use of distance-based weightings.
In this work we formalize the notion of local covariance by posing the problem in terms of recovering the optimal Gaussian covariance from a localized conditional distribution.

A related parallel effort~\citep{zucchelli2024bayesian}  developed an EnGMF where each kernel's covariance is dependent on the relative density of neighboring particles. All the k closest particles in the state-space equally contribute to the computation of the covariance of a kernel. With the aid of k-dimensional trees, the cost in the covariance generation is small. The resulting filter is used to track maneuvering spacecraft, whose transitional prior distributions are strongly non-Gaussian.

We propose a novel approach to localization that does not require an innate spatial structure in the state variables, but instead relies on some measure of inter-particle distance. 
Continuing the naming trend, we call our approach E-localization, standing for ensemble localization.
In this work we apply this methodology to kernel density estimation, creating the E-localized KDE (ELKDE) and to the EnGMF to beget the E-localized EnGMF (ELEnGMF).

We show that the proposed approach is able to more accurately describe the prior distribution with fewer samples, and that the resulting description of the prior leads to a better posterior. Results consist of both qualitative and quantitative analysis on a simple test problem. We additionally show that our approach begets better performance in the sequential filtering setting.

This work is organized as follows: we first provide some background on kernel density estimation, on the filtering problem, and on the EnGMF in~\cref{sec:background}. We next introduce the E-localization methodology in~\cref{sec:particle-localized-KDE}.
Penultimately, we showcase the ELKDE on a bivariate spiral example, and the ELEnGMF on the Lorenz '63 equations, in~\cref{sec:numerical-experiments}.
Finally, we have some concluding remarks in~\cref{sec:conclusions}.

\section{Background}
\label{sec:background}

We now introduce kernel density estimation (KDE) and the ensemble Gaussian mixture filter (EnGMF). 

\subsection{Kernel Density Estimation}
\label{sec:KDE}

Suppose that we have an ensemble $\*X_N = [x_1, \dots, x_N]$, of $N$ samples of dimension $n$ from the random variable $X$, which has the probability distribution $p_X$.

When we do not know $p_X$ directly, our goal is to approximate $p_X$ using the ensemble $\*X_N$ and all our other knowledge. We can build $\widehat p_X$, a kernel-based estimate of $p_X$, 
\begin{equation}\label{eq:full-KDE-estimate}
    \widehat{p}_{X}(x) = \frac{1}{N}\sum_{i = 1}^N K_{\Sigma_i}(x - x_i),
\end{equation}
where each $\Sigma_i$ is a positive definite matrix, and the kernels $K_{\Sigma_i}$ are functions that satisfy the following,
\begin{equation}
\begin{gathered}
    K_{\Sigma_i}(x) \geq 0, \forall x,\\
    \int_{\mathbb{R}^n} K_{\Sigma_i}(x)\mathrm{d}x = 1,\\
    \mathbb{E}[K_{\Sigma_i}(x)] = 0,\\
    0 \prec \Cov(K_{\Sigma_i}(x)) \prec \infty,\\
    K_{\Sigma_i}(x) = \lvert \Sigma_i\rvert^{-1/2} K_{I}\left(\Sigma_i^{-1/2}x\right),
\end{gathered}
\end{equation}
where $I$ is the identity matrix. The first two properties ensure that the kernels are  probability densities. The third ensures that they are mean zero, with the fourth ensuring that they have finite non-zero  covariance (in the SPD sense). The last property relates all the different parameterizations of the kernels to one another through the kernel with unit covariance.  

Note that there are $n(n-1)N + 1$ parameters of interest in the KDE estimate~\cref{eq:full-KDE-estimate}, namely the choice of kernel $K$, and the set of covariances $\{\Sigma_i\}_{i=1}^N$.

In this work we focus on the Gaussian kernel,
\begin{equation}\label{eq:Gaussian-kernel}
    K_{\Sigma_i}(x) =  (2 \pi)^{-n/2} \lvert\Sigma_i\rvert^{-1/2} e^{-\frac{1}{2} (x^T\Sigma_i^{-1}x)},
\end{equation}
as this makes the Kernel density estimate a Gaussian mixture model (GMM). GMMs are useful in state estimation as a GMM prior with a linear GMM observation leads to a GMM posterior~\citep{alspach1972nonlinear}.

Given knowledge of the covariance $\mathfrak{S}$ of the data, the canonical choice of the covariances in~\cref{eq:full-KDE-estimate}, is
\begin{equation}\label{eq:covariance-bandwidth-scaling}
    \Sigma_i = \beta^2_N \mathfrak{S},\quad i = 1,\dots, N,
\end{equation}
where $\beta^2_N$ is a scalar \textit{bandwidth} factor. The optimal bandwidth factor both depends on the choice of kernel $K$ and crucially on the underlying distribution of the data. As the underlying distribution is typically not known, the scaling factor that is optimal under Gaussian assumptions is typically chosen. 

The following result about the choice of optimal bandwidth is due to \citep{silverman2018density}:
\begin{lemma}
\label{thm:optimal-bandwidth}
    Without proof, if the distribution of interest, $p_X$ is Gaussian with mean $\mathfrak{m}$ and covariance $\mathfrak{S}$, and the kernel is Gaussian~\cref{eq:Gaussian-kernel}, then the bandwidth factor,
    \begin{equation}\label{eq:silvermans-rule}
        \beta^2_{N} = \left(\frac{4}{N(n + 2)}\right)^{\frac{2}{n + 4}},
    \end{equation}
    minimizes the mean integral squared error,
    \begin{equation}\label{eq:MISE}
    \mathbb{E}_{\mathbf{X}_N}\left[\int_{\mathbb{R}^n} \left(p_X(x) - \widehat{p}_{X}(x)\right)^2 \mathrm{d}x\right],
    \end{equation}
    between the true distribution and its KDE given by~\cref{eq:full-KDE-estimate}.
\end{lemma}

The approach used in~\cref{thm:optimal-bandwidth} we term the canonical KDE, or CKDE for short.

While the result in~\cref{thm:optimal-bandwidth} is satisfactory for Gaussian distributions, it does not provide us any intuition about how to choose the covariances in~\cref{eq:full-KDE-estimate} for non-Gaussian KDE.

One approach is to manually tune the bandwidth estimate~\cref{eq:silvermans-rule}, as it is known to typically be too large in the case of highly non-Gaussian data~\citep{silverman2018density, janssen1995scale}. A constant scaling factor,
\begin{equation}\label{eq:bandwidth-scaling-factor}
    \beta^2_N \xleftarrow[]{} s_\beta \beta^2_{N},
\end{equation}
with $0 < s_\beta$, can be applied to account for this deficiency. Tuning $s_\beta$ is difficult to perform correctly in a sequential filtering scenario, as a constant scaling factor may not be optimal for the wide range of possible distributions that can arise. Another alternative is to determine the scaling factor in an adaptive, online, fashion~\citep{popov2022adaptive} using the observation likelihood. The adaptive scaling of the covariances can still suffer from overfitting if the algorithm is not carefully tuned to the problem at hand.
A third alternative idea, called adaptive kernel density estimation (AKDE)~\citep{silverman2018density}, is to create a sequence of adaptive scaling factors,
\begin{equation}
\begin{gathered}
    \lambda_i = \left(\widehat{p}_{X}(x_i)/g\right)^{-\alpha},\\
    \text{where}\quad \log g = \frac{1}{N}\sum_{i=1}^N \log \widehat{p}_{X}(x_i),
\end{gathered}
\end{equation}
and $\alpha$ is a scaling factor typically taken to be $n^{-1}$, and every covariance is scaled by the factor
\begin{equation}
    \beta^2_i = \lambda_i^2\beta^2_{N}.
\end{equation}
Yet still, the AKDE idea suffers from the fact that the local curvature of the KDE estimate is determined by the global covariance of the whole distribution, and thus might be a poor approximation to the true non-Gaussian distribution.

\subsection{Ensemble Gaussian Mixture Filter}
\label{sec:EnGMF}

We first provide a brief overview of the state estimation/data assimilation problem.
Assume that we wish to estimate the true state $x^t$ of some physical process, but only have access to noisy non-linear observations of it,
\begin{equation}\label{eq:observation-operator}
    y = h(x^t) + \eta,
\end{equation}
where $\eta$ is unbiased additive Gaussian noise,
\begin{equation}
    \eta \sim \mathcal{N}(0, R),
\end{equation}
and $R$ is the observation error covariance.
Given a random variable $X^-$ that describes our prior knowledge about the state, we use Bayesian inference to find a random variable,
\begin{equation}\label{eq:Bayesian-inference}
    X^+ = \left.X^- \middle\vert \left(Y=y\right)\right.,
\end{equation}
that describes our posterior knowledge of the state, given the newfound observation~\cref{eq:observation-operator}.

We now describe the inner workings of the EnGMF~\citep{anderson1999monte,liu2016efficient,yun2022kernel,popov2022adaptive}, in order to understand some of its deficiencies.

The EnGMF operates in three discrete steps: (i) estimation of the prior density with an ensemble Gaussian mixture, (ii) update of the Gaussian mixture from prior to posterior, and (iii) resampling of exchangeable samples from the posterior. 

Given an ensemble of $N$ independently and identically distributed (iid)  samples from the prior distribution, $\*X^-_N = \begin{bmatrix}x^-_{1},\, \dots,\, x^-_{N}\end{bmatrix}$, using the KDE form~\cref{eq:full-KDE-estimate} with Gaussian kernel~\cref{eq:Gaussian-kernel} to construct an approximation of the prior distribution as a Gaussian mixture model (GMM),
\begin{equation}\label{eq:prior-GMM}
    \widehat{p}_{X^-}(x) = \frac{1}{N}\sum_{j=1}^N \,\mathcal{N}\left(x \,;\, x^-_{j},\, \widehat{\Sigma}^-_{j}\right),
\end{equation}
where each mean is centered at one of the ensemble members (particles), and the covariances $\widehat{\Sigma}^-_{j}$ are chosen in a way that attempts to make the distribution estimate more accurate.

The EnGMF formulas, used to construct the posterior GMM estimate, are,
\begin{equation}\label{eq:EnGMF-update}
    \begin{aligned}
        x^\sim_{j} &= x^-_{j} - \*G_{j}\left(h(x^-_{j}) - \*y\right),\\
        \widehat{\Sigma}^+_{j} &= \left(\*I -  \*G_{j}\*H_{i}^T\right)\widehat{\Sigma}^-_{j},\\
        \*G_{j} &= \widehat{\Sigma}^-_{j}\*H_{j}^T{\left(\*H_{j}\widehat{\Sigma}^-_{j}\*H_{j}^T + \*R_i\right)}^{-1},\\
        v_{j} &=  \frac{\mathcal{N}\left(\*y\,;\vert\, h(x^-_{j}),\, \*H_{j}\widehat{\Sigma}^-_{j}\*H_{j}^T + \*R\right)}{\sum_{j=1}^N\mathcal{N}\left(\*y\,;\, h(x^-_{j}),\, \*H_{j}\widehat{\Sigma}^-_{i,j}\*H_{j}^T + \*R\right)},\\
        \*H_{j} &= \left.\frac{d h}{d x}\right\rvert_{x = x^-_{j}},
    \end{aligned}
\end{equation}
where $x^\sim_{j}$ are the means of the posterior GMM, $\widehat{\Sigma}^+_{j}$ are the posterior covariances, $\*G_{j}$ are the particle gain matrices, $v_{j}$ are the weights of the posterior GMM modes,  and $\*H_{j}$ are the linearizations of the observation operator~\cref{eq:observation-operator} about the ensemble members (particles). When the observations are linear, this update formula is exact.

\begin{remark}
    Note that the gain matrices $\*G_j$ in~\cref{eq:EnGMF-update} are actually statistical gains, and as such are biased estimators~\citep{popov2020explicit} of the exact gains. Techniques such as ensemble inflation can be applied to compensate for this bias, though in practice, for the EnGMF with resampling, this issue has not shown to have a significant impact.
\end{remark}

\begin{remark}
    The mean and covariance updates in~\cref{eq:EnGMF-update} behave in a similar fashion to that of the extended Kalman filter. While this ensures that the EnGMF converges~\citep{popov2022adaptive} to `exact Bayesian inference', in the case of a non-linear model and non-Gaussian distribution, a poor description of the prior would lead to an even worse description of the posterior. An alternative formulation, using the unscented transform has been utilized by~\citep{reifler2021multi,reifler2023large}.
\end{remark}

Through the use of the mean and covariance updates in~\cref{eq:EnGMF-update}, the estimate of the posterior GMM can be constructed,
\begin{equation}\label{eq:posterior-GMM}
    \widehat{p}_{X^+}(x) = \sum_{j=1}^N v_{j} \,\mathcal{N}\left(x \,;\, x^\sim_{j},\, \widehat{\Sigma}^+_{j}\right),
\end{equation}
which is an approximation of the true posterior.

In the final step of the EnGMF, an ensemble of exchangeable samples from the posterior GMM~\cref{eq:posterior-GMM} is computed, resulting in $\*X^+_N = \begin{bmatrix}x^+_{1},\, \dots,\, x^+_{N}\end{bmatrix}$. There are many ways by which a Gaussian mixture model can be sampled~\citep{gilitschenski2014deterministic,liu2016efficient}. In this work we focus on the naive approach. First sample $N$ indices $j$ from the distribution defined by the weights $v_j$ from~\cref{eq:EnGMF-update}. For each index, take its corresponding Gaussian component in~\cref{eq:posterior-GMM} and sample a random variable from it, by the canonical procedure. 
For details on implementation of the resampling procedure for Gaussian mixture models see~\citep{liu2016efficient}.

\begin{remark}
    This resampling process does not generate iid samples from the approximate posterior distribution, as the posterior GMM weights, means, and covariances are all sample statistics, and are thus random variables derived from the prior ensemble.
    This means that samples from the posterior distribution are all conditioned on the prior ensemble, and the samples are therefore merely  exchangeable. In practice, the error of treating these samples as iid does not seem to make a significant impact when the ensemble size is large enough.
\end{remark}

The choice of covariance in the prior GMM~\cref{eq:prior-GMM} dictates the convergence of the EnGMF algorithm.
In the canonical EnGMF~\citep{anderson1999monte} the prior covariances are chosen as to align with~\cref{thm:optimal-bandwidth}, with the global covariance of the system estimated by the empirical covariance,
\begin{equation}\label{eq:global-sample-covariance}
    \widehat{\Sigma}^- = \frac{1}{N-1}\*X^-_N\left(\*I_N - \frac{1}{N}\*1_N\*1_N^T\right)\*X^{-,T}_N, 
\end{equation}
which is an unbiased estimator of the true covariance when the samples $\*X_N$ are iid.

In state estimation and data assimilation, the prior state $X^-$ is information sampled at a preceding step and propagated through a highly non-linear model, thus the distributions resulting from this propagation become highly non-Gaussian even if originally sampled from a Gaussian; consequently the empirical central second moment of the distribution~\cref{eq:global-sample-covariance} carries less of the total information than for Gaussian filters.

Though the EnGMF contains the word ``ensemble'', it is really a particle filter, though the distinction is not discussed further in this work.

\subsection{Motivation through a bimodal example}
\label{sec:bimodal-example}

We now show why the CKDE estimate with the Silverman bandwidth~\cref{thm:optimal-bandwidth} is a poor choice for some non-Gaussian distributions.

We construct a bivariate (in terms of two state variables $x = \begin{bmatrix} x_1 & x_2\end{bmatrix}^T$) Gaussian mixture model consisting of two constituent terms such that the local covariance of each constituent term in the mixture is significantly different from the global covariance of the total mixture. Each constituent term has the same local covariance,
\begin{equation}\label{eq:GM2-local-covariance}
    \mathfrak{S}_{\text{local}} = \begin{bmatrix} 1 & \rho\\\rho & 1\end{bmatrix},
\end{equation}
where $\rho$ is the correlation coefficient.
We separate the modes of the two Gaussian terms in the $x_2$ direction by a constant half-height $\upsilon$, setting the means of the two constituent terms to,
\begin{equation}
    \mathfrak{m}_1 = \begin{bmatrix}0 &
 \upsilon\end{bmatrix}^T,\quad \mathfrak{m}_2 = \begin{bmatrix}0 & -\upsilon\end{bmatrix}^T.
\end{equation}
Explicitly, the Gaussian mixture distribution that we analyze is,
\begin{equation}\label{eq:GM2-distribution}
    p_X(x) = \frac{1}{2}\left[\mathcal{N}(x ;\mathfrak{m}_1, \mathfrak{S}_{\text{local}}) + \mathcal{N}(x ;\mathfrak{m}_2, \mathfrak{S}_{\text{local}})\right],
\end{equation}
where each term is equally weighted.

By a simple application of known formulas~\citep{GMcovarianceStackOverflow} the global covariance of~\cref{eq:GM2-distribution} is given by,
\begin{equation}\label{eq:GM2-global-covariance}
    \mathfrak{S}_{\text{global}} = \begin{bmatrix} 1 & \rho\\\rho & 1 + \upsilon^2\end{bmatrix},
\end{equation}
meaning that as we increase the separation half-height $\upsilon$, the second state variable starts to dominate the global covariance.

We can formalize this intuition by looking at the dominant eigenvectors of the covariances.
When the correlation coefficient $\rho$ is positive, the dominant eigenvector of the local covariance~\cref{eq:GM2-local-covariance} is,
\begin{equation}\label{eq:local-dominant-eigenvector}
    \begin{bmatrix}
1 & 1    
\end{bmatrix}^T,
\end{equation}
which is not dependent on either the correlation coefficient $\rho$ or the separation half-height $\upsilon$.
The dominant eigenvector of the global covariance~\cref{eq:GM2-global-covariance}, on the other hand, is,
\begin{equation}\label{eq:global-dominant-eigenvector}
        \begin{bmatrix}
\frac{-\upsilon^2 + \sqrt{\upsilon^4 + 4 \rho^2}}{2\rho} & 1    
    \end{bmatrix}^T,
\end{equation}
which tends towards an eigenvector of,
\begin{equation}
    \begin{bmatrix}
0 & 1    
\end{bmatrix}^T,
\end{equation}
as the separation half-height $\upsilon$ tends towards infinity. The eigenvalue corresponding to this eigenvector is,
\begin{equation}
    \frac{1}{2} \left(\sqrt{\nu ^4+4 \rho ^2}+\nu ^2+2\right),
\end{equation}
which tends towards infinity as the separation half-height $\nu$ increases. 
This creates a large discrepancy between the global covariance~\cref{eq:GM2-global-covariance}, and the local covariance~\cref{eq:GM2-local-covariance}. 
Thus, the CKDE approximation induced by the bandwidth scaling of the global covariance~\cref{eq:covariance-bandwidth-scaling} fails to capture the local density for a sufficiently large $\nu$.
In other words, a distribution that is representable in terms of a Gaussian mixture with finite covariances would degenerate to not being representable using the CKDE method.

From this example, it is clear that a method of determining the local covariance structure of the distribution is needed.

\section{Ensemble-localized Kernel density estimation}
\label{sec:particle-localized-KDE}

When the distribution of interest $p_X$ is non-Gaussian, the global covariance can create a poor approximation of the local relationship of the data. 
Instead of looking for the optimal bandwidth like in~\cref{thm:optimal-bandwidth}, we go back to the fundamentals of KDE estimates~\cref{eq:full-KDE-estimate}, and attempt to find covariance matrices $\{\Sigma_i\}_{i=1}^N$ that approximate the distribution locally around each particle.

We focus on attempting to understand the local covariance structure of the distribution, $p_X$, around a sample, $x_i$, therefrom. Essentially, we attempt to answer the following question:
if the local behavior of the distribution around the point $x_i$ extended globally, what would the covariance of the resulting distribution be?

In order to attempt to answer this question, we introduce the synthetic random variable $\mathcal{L}$ with the same support as $X$ and whose conditional distribution $\mathcal{L}|X=x$ is the description of the local information around the point $x$, and attempt to extract the local covariance around $x$ through understanding the distribution of $X | \mathcal{L} = x_i$ which is the distribution of $X$ conditioned on the local information around the realization $x_i$. We term this idea `E-localization', and its application to kernel density estimation, E-localized KDE or `ELKDE' for short.

From Bayes' rule, we can see that,
\begin{equation}
    p_{X | \mathcal{L} = x_i }(x)  \propto p_{\mathcal{L} | X = x }( x_i ) \ p_X(x),
\end{equation}
thus, the knowledge contained in the distribution of $p_{\mathcal{L} | X = x }$  is sufficient to explain all the additional, local, information.

In this work we explore the following choice:
\begin{equation}\label{eq:Gaussian-localization-choice}
    p_{\mathcal{L} | X = x}(\xi) = C e^{-\frac{1}{2} ((\xi - x)^T S_\xi^{-1}(\xi - x))},
\end{equation}
where $C$ is a normalizing constant and the symmetric positive definite matrix $S_\xi$ is an application-specific tunable parameter that is unique to every $\xi$ around which we wish to localize. This means, in effect that a different choice of~\cref{eq:Gaussian-localization-choice} is made for every individual particle.

Our goal is to find a covariance matrix, $\widetilde{\Sigma}_i$, that represents the local covariance structure of $p_X$ around the point $x_i$.
In the case when the random variable $X$ is Gaussian, the ELKDE method should exactly recover the CKDE method, providing equivalence and some limited sense of optimality by~\cref{thm:optimal-bandwidth}.

Thus, we first look at the case when $p_X$ is Gaussian with known mean and covariance.
If the distribution is Gaussian, its local behavior around every single point has to be Gaussian as well. Thus, it should be possible to recover the global covariance  
from the information contained in $p_{X|\mathcal{L}=x_i}$, for any choice of the parameter $S_\xi$ in~\cref{eq:Gaussian-localization-choice}.

In the subsequent proof, we show that by only knowing information about $p_{X|\mathcal{L}}$ and the covariance 
$S_{\xi}$ from~\cref{eq:Gaussian-localization-choice}, it is possible to recover the covariance of the information contained within $p_X(x)$, with the knowledge of $X$ is Gaussian.

\begin{theorem}[Approximate Gaussian covariance with E-localization]\label{thm:ELKDE-Gaussian-covariance}
    Take the random variable $X$ to be Gaussian,
    \begin{equation}
        p_X(x) = \mathcal{N}(x;\mathfrak{m}, \mathfrak{S}),
    \end{equation}
    with $p_X$ being the distribution, $\mathfrak{m}$ the mean, and $\mathfrak{S}$ the covariance.
    Take also the ensemble of $N$ samples $\mathbf{X} = [x_1, \dots, x_N]$ therefrom.

    If the distribution of the conditional realizations $\{x_i\}_{i=1}^N$ are as in~\cref{eq:Gaussian-localization-choice},
    \begin{equation}
        p_{\mathcal{L} | X = x }(x_i) = \frac{C}{ \lvert S_i\rvert^{\frac{1}{2}}} e^{-\frac{1}{2} ((x_i - x)^T S_i^{-1}(x_i - x))},
    \end{equation}
    with matrices $\{S_i\}_{i=1}^N$, standing for $\{S_{x_i}\}_{i=1}^N$, then the  ensemble localized density estimate, 
    \begin{equation}
        \widehat{p}_{X}(x) = \frac{1}{N}\sum_{i=1}^N K_{\widehat{\Sigma}_i}(x - x_i),
    \end{equation}
    with covariances,
    \begin{equation}\label{eq:ELKDE-Gaussian-covariance}
    \begin{aligned}
        \widehat{\Sigma}_i &= \beta^2_N\widetilde{\Sigma}_i,\\
        \widetilde{\Sigma}_i&=\Cov(X|\mathcal{L}\!=\!x_i){\left[S_i - \Cov(X|\mathcal{L}\!=\!x_i)\right]}^{-1} S_i,
    \end{aligned}
    \end{equation}
    where $i = 1, \dots N$, is exactly equivalent to the canonical kernel density estimate from~\cref{thm:optimal-bandwidth} for any choice of $\{S_i\}_{i=1}^N$.
\end{theorem}
\begin{proof}
    It suffices to show that for all $i$, every local covariance is equal to the global Gaussian covariance,  $\widetilde{\Sigma}_i = \mathfrak{S}$.
    From the Kalman filter equations, it is known that,
    \begin{gather*}
            \Cov(X | \mathcal{L}\!=\!x_i) = \mathfrak{S} - \mathfrak{S}{\left(\mathfrak{S} + S_i\right)}^{-1}\mathfrak{S},\\
            \begin{aligned}
            \phantom{S}&= \mathfrak{S}\ {\left(\mathfrak{S} + S_i\right)}^{-1} {\left(\mathfrak{S} + S_i\right)}- \mathfrak{S}{\left(\mathfrak{S} + S_i\right)}^{-1}\mathfrak{S},
            \\
            &= \mathfrak{S}{\left(\mathfrak{S} + S_i\right)}^{-1}S_i,
        \end{aligned}
    \end{gather*}
    therefore the covariances are,
    \begin{align*}
        \widetilde{\Sigma}_i &= \mathfrak{S}{\left(\mathfrak{S} + S_i\right)}^{-1}S_i\left[S_i - \mathfrak{S}{\left(\mathfrak{S} + S_i\right)}^{-1}S_i\right]^{-1}S_i,\\
        &= \mathfrak{S}{\left(\mathfrak{S} + S_i\right)}^{-1}S_i\left[S_i^{-1} + S_i^{-1}\mathfrak{S}S_i^{-1}\right]S_i = \mathfrak{S},
    \end{align*}
    independent of $S_i$, as required.
\end{proof}

It is also of note that in the Gaussian case $\Cov(X|\mathcal{L}\!=\!x_i)$ does not depend on the outcome $x_i$, but only on the covariance parameter $S_i$ which is not generally true in the non-Gaussian case. In the general case, we make the notational definition,
\begin{equation}\label{eq:covariance-X-given-L-eq-xi}
    \mathfrak{C}(x_i, S_i) \coloneqq \Cov(X | \mathcal{L}\!=\!x_i).
\end{equation}

We now look at what happens when the distribution $p_X$ is non-Gaussian. If the covariance $S_i$ from~\cref{eq:Gaussian-localization-choice} tends towards infinity, then the local covariance from~\cref{eq:ELKDE-Gaussian-covariance} should tend towards the global covariance of the distribution of $X$.

\begin{corollary}[Local covariance becomes global]\label{cor:local-covariance-becomes-global}
    When the Gaussian distribution assumptions in~\cref{thm:ELKDE-Gaussian-covariance} are relaxed and $X$ is taken from some arbitrary non-Gaussian distribution, the local covariance matrix $\widetilde{\Sigma}_i$ from~\cref{eq:ELKDE-Gaussian-covariance}, approaches the global covariance as the synthetic covariance $S_i$ from~\cref{eq:Gaussian-localization-choice} approaches infinity,
    \begin{equation}
        \lim_{\lvert S_i\rvert \to\infty}\widetilde{\Sigma}_i = \Cov(X),
    \end{equation}
    where $\Cov(X)$ is the global covariance of the (now) non-Gaussian random variable $X$, and $\lvert S_i\rvert$ is some notion of generalized variance.
\end{corollary}

From this we can glean that the larger the synthetic covariance $S_i$ gets, the less local the information becomes, and the covariance approaches the global estimate. 
It is therefore reasonable to ask if the converse is also true: as the covariance parameter $S_i$ approaches zero, does the covariance in~\cref{eq:ELKDE-Gaussian-covariance} tend towards a local description of the covariance around the point $x_i$?

\begin{remark}[Generalized variance]\label{rem:generalized-variance}
    It is of note that generalized variance in \cref{cor:local-covariance-becomes-global} can take multiple forms. Common choices include,
    \begin{equation}
        \lvert S_i \rvert = \det S_i,
    \end{equation}
    which is the determinant of the matrix $S_i$, and,
    \begin{equation}
        \lvert S_i \rvert = \operatorname{tr} S_i,
    \end{equation}
    which is its trace.
    The choice of generalized variance is ancillary to the main results of this work.
\end{remark}

It is of note that the covariance~\cref{eq:ELKDE-Gaussian-covariance} in~\cref{thm:ELKDE-Gaussian-covariance} is only a good estimate in the case that the distribution of $X|\mathcal{L}$ is actually Gaussian. 
It is possible---either exactly, or numerically---that,
\begin{equation}
    S_i \not> \mathfrak{C}(x_i, S_i),
\end{equation}
which would make the covariance~\cref{eq:ELKDE-Gaussian-covariance} not positive definite, thus, in order to generalize~\cref{thm:ELKDE-Gaussian-covariance} to non-Gaussian distributions, an alternate formulation of the covariance must be chosen.

Utilizing the information we have gleaned from~\cref{thm:ELKDE-Gaussian-covariance}, from~\cref{cor:local-covariance-becomes-global}, and from the above discussion, we propose a definition of the local covariance matrix for generic, non-Gaussian distributions:
%
\begin{definition}[Local covariance]
\label{def:local-covariance}
    Given a random variable $X$ with finite first two moments, and given a Gaussian random variable $\mathcal{L}|X\!=\!x_i$ that encapsulates our local uncertainty around $x_i$ with covariance $S_i$, from~\cref{eq:Gaussian-localization-choice}, the covariance that extends the local behavior of the distribution globally around the point $x_i$ is given by,
    \begin{equation}\label{eq:local-covariance-optimization}
    \begin{gathered}
        \widetilde{\Sigma}_i \coloneqq \widetilde{\Sigma}(x_i, S_i) = \mathfrak{C}(x_i, S_i) {\left[S_i - \mathfrak{C}(x_i, S_i)\right]}^{-1}\! S_i,\\
        \text{where}\quad \left\{\begin{aligned}
            S_i &= \argmin_{S\in\text{SP}(n)} \lvert S\rvert,\\\text{such that}&\quad\widetilde{\Sigma}(x_i, S_i) \in \text{SP}(n), 
        \end{aligned}\right\},
        \end{gathered}
    \end{equation}
    meaning that the local covariance $\widetilde{\Sigma}_i$ is the covariance that minimizes its argument $S_i$ while remaining in the set of all $n$-dimensional symmetric positive definite matrices, $\text{SP}(n)$.
\end{definition}

\Cref{def:local-covariance} is one attempt at extending~\cref{thm:ELKDE-Gaussian-covariance} to non-Gaussian distribution through the same reasoning as in~\cref{cor:local-covariance-becomes-global}. However, it is not the only possible choice of extension. The optimality of formulation proposed therein is of relevant interest but beyond the scope of this work.

Note that when the random variable $X$ is Gaussian, the proposed sense of local covariance in~\cref{def:local-covariance} is exactly equivalent to the derivation in~\cref{thm:ELKDE-Gaussian-covariance}.

\begin{remark}[Using $\mathfrak{C}(x_i, S_i)$]
    In~\citep{popov2023elengmf}, on which this work is based, a scaling of the conditional covariance matrix
    \begin{equation}
        \Sigma_i = s_{\mathcal{L},i}\mathfrak{C}(x_i, S_i),
    \end{equation}
    where $s_{\mathcal{L},i}$ is a scaling factor determined for each covariance,
    was used for the covariance estimates. As there are no formal guarantees on when such a covariance is optimal, this idea is not further explored in this work.
\end{remark}

\subsection{Revisiting the bimodal example}

\begin{figure}
    \centering
    \includegraphics[width=0.49\linewidth]{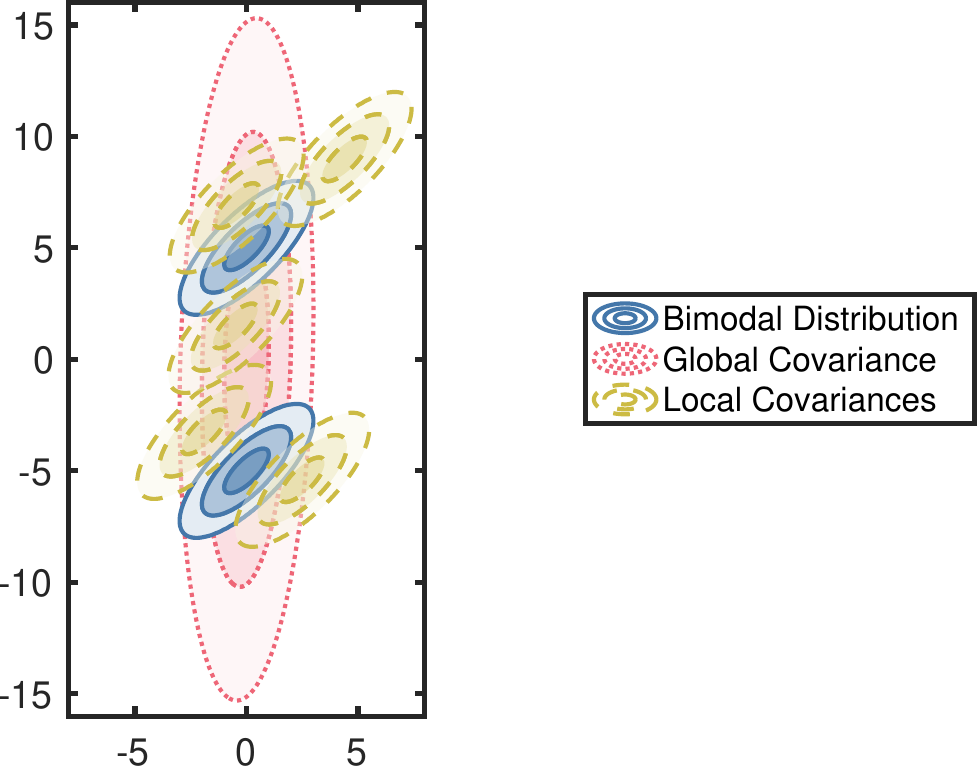}
    \caption{A demonstration of the local covariance technique~\cref{def:local-covariance} on the bimodal example~\cref{sec:bimodal-example}. The blue ellipses with solid lines represent the 1, 2, and 3-$\sigma$ bounds on the true bimodal distribution. The large red vertical ellipses with dotted outlines represent the global covariance. The yellow ellipses with dashed outlines represent the local covariances centered at hand-sampled points from the true distribution.}
    \label{fig:bimodal-local-covariance}
\end{figure}

We now revisit the bimodal example in~\cref{sec:bimodal-example}, and show that the proposed ELKDE method is robust to increase in the separation half-width parameter $\nu$.

Take a sample $x_i$ from the distribution defined by~\cref{eq:GM2-distribution}, then the distribution of $X$ given $\mathcal{L} = x_i$, would be defined by,
\begin{equation}
\begin{aligned}
    p_{X |\mathcal L=x_i}(x) = &\phantom{+}\,\,\, u_1\mathcal{N}(\ell_1 , \mathfrak{S}_{\text{local}}(\mathfrak{S}_{\text{local}} + S_i)^{-1}S_i)\\
    &+ u_2\mathcal{N}(\ell_2 , \mathfrak{S}_{\text{local}}(\mathfrak{S}_{\text{local}} + S_i)^{-1}S_i),
    \end{aligned}
\end{equation}
where the means $\ell_1$ and $\ell_2$ are of no consequence.

As $\nu\to\infty$, the sample almost surely comes from the mode whose mean is closest. Without loss of generality, assume that this is the first mode, thus the weights collapse, $u_1 \to 1$ and $u_2 \to 0$, which can be seen from the EnGMF equations~\cref{eq:EnGMF-update}.
Therefore in this case the distribution collapses to a simple Gaussian, thus by~\cref{thm:ELKDE-Gaussian-covariance}, the covariance around the point $x_i$ by the ELKDE method is a scaling of $\mathfrak{S}_{\text{local}}$, and not of $\mathfrak{S}_{\text{global}}$ as by the CKDE method.

Moreover, in the case of a fixed $\nu = 5$, for any point in the distribution, the local covariance as computed by the optimization problem in~\cref{def:local-covariance}, is always the local covariance. 
This can be visually seen in~\cref{fig:bimodal-local-covariance} where the bimodal distribution, the global covariance and the local covariances at select points are plotted.
As can be seen, the local covariances produced by~\cref{def:local-covariance}, with determinant generalized variance as a metric, are identical to the covariances of the bimodal distribution, and do not resemble the global covariance.

We thus show that the ELKDE method is robust to the degenerate behavior shown in~\cref{sec:bimodal-example}.

\subsection{Practical Implementation}

There are several considerations to be made before a practical implementation of the covariance estimate in~\cref{def:local-covariance} can be attempted.

First we have to consider how to compute the covariance matrix $\mathfrak{C}(x_i, S_i)$, from \cref{eq:covariance-X-given-L-eq-xi}, given only the ensemble $\*X$. We begin by defining local weights,
\begin{equation}\label{eq:weights}
    w_{i,j} \propto \mathcal{N}(x_i ; x_j, S_i)
\end{equation}
such that the column vector of weights $\*w_i$ defines the importance of each  ensemble member given the random variable $\mathcal{L}=x_i$. The covariance can then be estimated from the ensemble and the importance weights by the unbiased estimator,
\begin{equation}\label{eq:weighted-covariance}
    \mathfrak{C}(x_i, S_i) \approx \frac{1}{1 - \mathbf{w}_i^T\mathbf{w}_i}\left[\mathbf{X}\left(\bdiag \mathbf{w}_i - \mathbf{w}_i\mathbf{w}_i^T\right)\mathbf{X}^T\right],
\end{equation}
where $\bdiag$ is the vector to diagonal matrix operator.

A practical numerical consideration is the problem of degenerate weights, meaning that the covariance estimate in~\cref{eq:weighted-covariance} can be numerically singular. A simple way to overcome this is by slightly nudging the weights towards uniformity,
\begin{equation}
    w_{i, j} \leftarrow (1 - \alpha) w_{i, j} + \frac{\alpha}{N},
\end{equation}
where the factor $\alpha$ is chosen to be small. In this work we fix $\alpha = 10^{-4}$.

The local covariance in~\cref{def:local-covariance} requires the solution of a matrix-valued constrained optimization problem, which is practically useful.
We can approximate the local covariance optimization~\cref{eq:local-covariance-optimization} by the simpler formula,
\begin{equation}\label{eq:local-covariance-simple}
        \widetilde{\Sigma}_i \!=\!\! \lim_{\lvert S_i\rvert\to0}\!\Pi_{\text{SP}(n)}\!\left(\!\mathfrak{C}(x_i, S_i){\left[S_i - \mathfrak{C}(x_i, S_i)\right]}^{-1}\! S_i\right),
\end{equation}
where $\Pi_{\text{SP}(n)}$ is the projection of a matrix onto the set of $n$-dimensional symmetric positive definite (SPD) matrices.

The next consideration is the choice of covariance matrix $S_i$. If our ensemble $\*X$ is finite, then it is impractical to choose $S_i$ to be arbitrarily small, thus the aim is to find a formula for $S_i$ that tends towards zero as the ensemble size grows, but still encapsulates numerically useful information about the distribution from the finite ensemble.

As our aim is to limit the impact of far away samples and only consider the impact of local samples, we choose,
\begin{equation}\label{eq:choice-of-Si}
    S_i = r_i^2 I,
\end{equation}
where the scalar `radius', $r_i$ is chosen such that the weights~\cref{eq:weights} are meaningful in some sense. To that end we again take inspiration from Kernel density estimation literature~\citep{silverman2018density} and from the ensemble transform particle filter~\citep{reich2013nonparametric} in determining the localization radius. For the $i$th particle, take the distance to the $\sqrt{N}$th (rounded) nearest neighbor according to some distance $d(i,k)$, and call it $d_i$.
This idea is loosely based on the use of inter-particle distance in the ensemble transform particle filter (ETPF)~\citep{reich2013nonparametric}, though in the ETPF the inter-particle distance is used to solve an optimal transport problem.
We take the localization radius to be a scaling of this distance,
\begin{equation}\label{eq:scaled-radius}
    r_i = s_r d_i,
\end{equation}
where $s_r$ is the radius scaling factor (in this work taken to be $s_r = 1$). Thus the covariance $S_i$ should converge to $0$ as $N\to\infty$ in probability, as long as the distribution of interest has connected support. The above procedure is a modification of the one used in~\citep{zucchelli2024bayesian}.

\begin{remark}
    Note that taking $r_i\to0$ is equivalent to taking any measure of generalized variance~\cref{rem:generalized-variance} to zero, meaning that $\lvert S_i \rvert\to0$ as in  \cref{eq:local-covariance-simple}.
\end{remark}

\begin{remark}[Using $S_i$ as the local covariance]
An immediately obvious question to ask, is why not just use $S_i$ for the covariances in~\cref{eq:full-KDE-estimate} instead, such as,
\begin{equation} \label{eq:alt-GMM}
    \widehat{p}(x) = \sum_{i=1}^N \mathcal N(x;x_i,S_i).
\end{equation}
While this estimate would converge to the exact distribution from which the samples originate given that the distribution has support over all real space, it would not converge to the optimal estimate in~\cref{thm:optimal-bandwidth}.
Additionally, the distribution in~\cref{eq:alt-GMM} would not even account for the global scaling of the distribution, as the covariance would be diagonal.
\end{remark}

The final consideration is finding a useful projection operator $\Pi_{\text{SP}(n)}$ in~\cref{def:local-covariance}. In order to see the two different ways in which this projection can be performed, it is useful to think about the logarithm of the covariance in~\cref{eq:ELKDE-Gaussian-covariance},
\begin{align}
    \log \widetilde{\Sigma}_i &= \log\left[\Cov(X|\mathcal{L}=x_i){\left[S_i - \Cov(X|\mathcal{L}=x_i)\right]}^{-1} S_i\right], \label{eq:log-full} \\
    &\begin{multlined}
        =\log\Cov(X|\mathcal{L}=x_i) - \log\left[S_i - \Cov(X|\mathcal{L}=x_i)\right]\\ +  \log S_i,\label{eq:log-partial}
    \end{multlined}
\end{align}
as the logarithm of the covariance merely has to be a symmetric matrix. 
There are two key elements, first ensuring that the final covariance estimate is SPD, which is equivalent to ensuring that~\cref{eq:log-full} is symmetric and real, and the second is ensuring that certain key constituent covariances in the calculation themselves are SPD, which is equivalent to ensuring that every matrix logarithm in~\cref{eq:log-partial} is symmetric and real.

Let the operator on the arbitrary symmetric matrix $A$,
\begin{equation}
    \mathcal{E}_{\epsilon} A,
\end{equation}
define the operation that bounds the eigenvalues of $A$ by $\epsilon$ from below.
We can first think about ensuring that the matrix~\cref{eq:log-full} is symmetric and real, then
one possible projection operator is,
\begin{equation}\label{eq:naive-eigenvalue-bound}
    \Pi_{\text{SP}(n)} = \mathcal{E}_{\epsilon},
\end{equation}
which simply eliminates possible negative eigenvalues.

If we instead concentrate on ensuring that the logarithms in~\cref{eq:log-partial} are symmetric and real, then one possible projection operator is,
\begin{equation}\label{eq:smart-eigenvalue-bound}
    \Pi_{\text{SP}(n)}(\cdot) = \mathcal{E}_{\epsilon_1}\left[\Cov(X|\mathcal{L}){\left[\mathcal{E}_{\epsilon_2}(S_i - \Cov(X|\mathcal{L}))\right]}^{-1} S_i\right],
\end{equation}
which bounds the eigenvalues not only of the matrix, but of the constituent terms as well.

We provide an algorithmic overview of the E-localization procedure for finding the local covariance matrices in the ELKDE in~\cref{alg:ELKDE}.

\begin{algorithm}[t]
 \caption{Local covariance estimation in the ELKDE}
 \label{alg:ELKDE}
 \begin{algorithmic}[1]
 \renewcommand{\algorithmicrequire}{\textbf{Input:}}
 \renewcommand{\algorithmicensure}{\textbf{Output:}}
 \REQUIRE  Ensemble $\*X_N$, projection function $\Pi_{\text{SP}(n)}$
 \ENSURE  E-localized covariances $\{\Sigma_i\}_{i=1}^N$
  \FOR {$i = 1$ to $N$}
  \STATE{\% Find the distances between $x_i$ and other samples}
  \FOR {$j = 1$ to $N$}
  \STATE $d_{ij} \xleftarrow[]{} \lVert{x_i - x_j}\rVert$
  \ENDFOR
  \STATE {$r_i \xleftarrow[]{} \sqrt{N}$th sorted $d_{ij}$}
  \STATE {$S_i \xleftarrow[]{} r_i^2 I_n$}
  \STATE {\% Find the weights of the local information}
  \FOR {$j = 1$ to $N$}
  \STATE {$a_{ij} \xleftarrow[]{} \log\mathcal{N}(x_j ; x_i, S_i)$}
  \ENDFOR
  \STATE{\% Numerically stable algorithm for computing the weights}
  \STATE{$\*w_{i} \xleftarrow[]{} \exp\left(\*a_i - \logsumexp(\*a_i)\right)$}
  \STATE{$\Cov(X|\mathcal{L}) \xleftarrow[]{}$ by \cref{eq:weighted-covariance}}
  \STATE {\% Compute the unprojected estimate of the covariance}
  \STATE{$\widetilde{\Sigma}_i\xleftarrow[]{}\Cov(X|\mathcal{L}){\left[S_i - \Cov(X|\mathcal{L})\right]}^{-1} S_i$}
  \STATE {\% Compute the projected and scaled covariance}
  \STATE{$\Sigma_i = \beta^2_N\Pi_{\text{SP}(n)}(\widetilde{\Sigma}_i)$}
  \ENDFOR
 \end{algorithmic}
 \end{algorithm}


\section{Numerical Experiments}
\label{sec:numerical-experiments}

The goal of our numerical experiments is to first showcase the utility of the ELKDE approach on a simple-to-visualize data set, for which we use a simple bivariate spiral.  Our second experiment, utilizing a more complicated chaotic dynamical system fulfils the second goal, to showcase the ELKDE methodology on the EnGMF, and thus showcase its utility for sequential filtering.

\subsection{A bivariate spiral}

\begin{figure}[tp]
    \centering
    \includegraphics[width=0.49\linewidth]{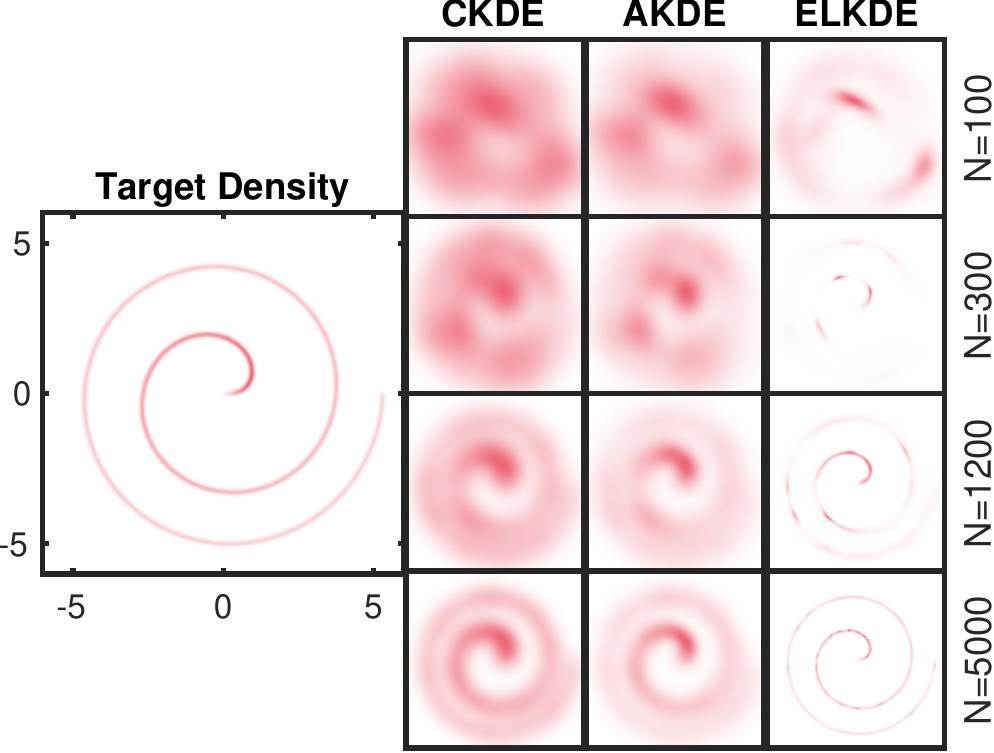}
    \caption{A qualitative look at the accuracy of the different KDE methods explored in this work for the continuous Gaussian mixture model described in~\cref{eq:continuous-GMM-example}. The true distribution is shown on the left, the canonical KDE method is labeled by `CKDE', the adaptive KDE method is labeled as `AKDE', and the E-localized KDE method is labeled as `ELKDE'.}
    \label{fig:spiral-figure}
\end{figure}

\pgfplotsset{clean/.style={axis lines*=left,
        axis on top=true,
        axis x line shift=0.0em,
        axis y line shift=0.75em,
        every tick/.style={black, thick},
        axis line style = ultra thick,
        tick align=outside,
        clip=false,
        major tick length=4pt}}

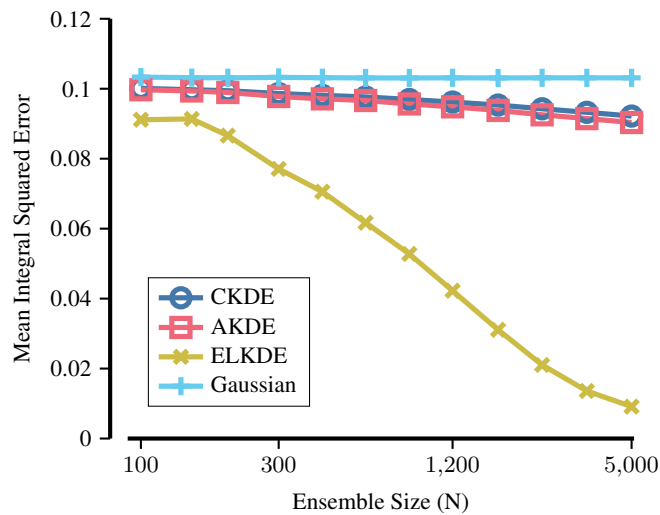
\begin{figure}[tp]
    \centering
    \resizebox {0.5\textwidth} {!} {
    \begin{tikzpicture}
    \begin{semilogxaxis}[clean,
        cycle list name=tol,
        xmode=log,
        log ticks with fixed point,
        xtick={100, 300, 1200, 5000},
        table/col sep=comma,
        xmin = 90,
        xmax = 5100,
        ymin = 0,
        ymax = 0.12,
        ytick={0,0.02, 0.04,0.06, 0.08,0.1, 0.12},
        yticklabels={0,0.02, 0.04,0.06, 0.08,0.1, 0.12},
        xlabel = {Ensemble Size (N)},
        ylabel = {Mean Integral Squared Error},
        every axis plot/.append style={line width=2pt, mark size=3.5pt},
        legend style={at={(0.2,0.23)},anchor=center},
        legend cell align={left}]
    \addplot table [x=N, y=KDE, col sep=comma] {data/nonGMISEresults.csv};
    \addlegendentry{CKDE };
    \addplot table [x=N, y=AKDE, col sep=comma] {data/nonGMISEresults.csv};
    \addlegendentry{AKDE };
    \addplot table [x=N, y=LKDE, col sep=comma] {data/nonGMISEresults.csv};
    \addlegendentry{ELKDE };
    \addplot table [x=N, y=ND, col sep=comma] {data/nonGMISEresults.csv};
    \addlegendentry{Gaussian };
    \end{semilogxaxis}
    \end{tikzpicture}
    }
    \caption{A quantitative look at the accuracy of the different KDE methods explored in this work for the continuous Gaussian mixture model described in~\cref{eq:continuous-GMM-example} of ensemble size versus error in terms of MISE. The canonical KDE method is labeled by `CKDE', the adaptive KDE method is labeled as `AKDE', and the E-localized KDE method is labeled as `ELKDE'. Additionally, the approximation by the empirical Gaussian distribution is provided as a baseline.}
    \label{fig:bivariate-experiment}
\end{figure}

Our first numerical experiment is meant to test the accuracy of the ELKDE in a highly non-Gaussian setting.

Consider the following distribution of continuous Gaussian mixture~\citep{zucchelli2023gif} form,
\begin{equation}\label{eq:continuous-GMM-example}
    p_X(x) = \frac{1}{4\pi}\int_0^{4\pi} \mathcal{N}(x ; \mathfrak{m}(z), \sigma^2 I_2)\,\mathrm{d} z,
\end{equation}
where $I_2$ is the bivariate identity matrix, and the constituent mean,
\begin{equation}
    \mathfrak{m}(z) = \frac{3}{2}\sqrt{z}\begin{bmatrix} \cos(z) \\ \sin(z)\end{bmatrix},
\end{equation}
represents a Fermat spiral~\citep{lawrence2013catalog}, and the small covariance, defined by $\sigma^2 = 2^{-8}$, represents the local covariance of each one of the infinitesimal modes.

The mean of this distribution is given by the closed form expression,
\begin{equation}
    \mathbb{E}[X] = \frac{3}{8\sqrt{2\pi}}\begin{bmatrix}
        -S\!\left(2 \sqrt{2}\right)\\
   C\!\left(2 \sqrt{2}\right)-2\sqrt{2}
    \end{bmatrix}\approx \begin{bmatrix}
        -0.06\\-0.35
    \end{bmatrix},
\end{equation}
where $S$ and $C$ are the Fresnel functions~\citep{olver2010digital},
and the covariance $\Cov(X) = \mathfrak{S}$ is,
\begin{equation}
\begin{aligned}
    \mathfrak{S}_{1,1} &=
    \frac{9 \pi }{4}-\frac{9 S\!\left(2 \sqrt{2}\right)^2}{128 \pi }+\frac{1}{256}\\ 
    \mathfrak{S}_{2,1} &=-\frac{9 \left(\left(2 \sqrt{2}-C\!\left(2 \sqrt{2}\right)\right) S\!\left(2
   \sqrt{2}\right)+8 \pi \right)}{128 \pi }\\
   \mathfrak{S}_{2,2} &= \frac{9 \left(4 \sqrt{2} C\left(2 \sqrt{2}\right)-C\left(2 \sqrt{2}\right)^2-8+32 \pi
   ^2\right)}{128 \pi }\\&\phantom{=}+\frac{1}{256}
\end{aligned}
\end{equation}
which is approximately,
\begin{equation}
    \mathfrak{S} \approx \begin{bmatrix}7.07&-0.58\\-0.58&6.95\end{bmatrix}.
\end{equation}

This distribution is particularly challenging as it requires an increasing number of of Gaussian mixture terms to approximate accurately, as it cannot be fully represented by a finite Gaussian mixture. Additionally, there is three orders of magnitude separating the local and global covariance terms, thus the CKDE and AKDE techniques should have particular difficulty with this problem.

We present results on various amount of samples $N$ ranging from $N=100$ to $N=5000$, and make use of the MISE~\cref{eq:MISE} as a measure of error. We compute the MISE over 12 Monte Carlo runs for each ensemble size $N$. The samples from the spiral distribution~\cref{eq:continuous-GMM-example} are computed exactly, while the MISE computation uses a uniformly spaced approximation to the integral with 10000 evenly spaced points of $z$.
We compare the CKDE, the AKDE, and the ELKDE methods. 

For this experiment we make use of the second projection method from~\cref{eq:smart-eigenvalue-bound}, with parameters $\epsilon_1 = 10^{-4}$ and $\epsilon_2 = 10^{-2}$ as this method produces covariance estimates that are in line with performing the optimization procedure in~\cref{def:local-covariance}.

A qualitative look at the accuracy of the methods can be seen in~\cref{fig:spiral-figure}. It can be seen that the CKDE method suffers from the fact that the global estimate of the covariance is not appropriate for approximating the target probability density. The AKDE method seems to be able to better resolve the form of the spiral, though not significantly better than the CKDE method. The ELKDE method appears to be able to fully resolve the correct spiral with $N=1200$ and $N=5000$ samples, with a partially resolved spiral visible for as little as $N=300$ samples.

For the quantitative experiment, we additionally look at approximating the distribution by the empirical Gaussian as a baseline. As can be seen in~\cref{fig:bivariate-experiment}, the CKDE is a slight improvement over the Gaussian approximation, with the AKDE having  a noticeable, but marginal impact on the error of the CKDE method, though both global methods barely do better than the Gaussian baseline.
The ELKDE method, on the other had, shows significant improvement, especially in the slope of the improvement with respect to the number of samples. For $N=5000$, the ELKDE method results in a MISE of one order of magnitude less than the other methods.

\subsection{Lorenz '63 sequential filtering example}

\begin{figure}[tp]
    \centering
    \begin{tikzpicture}
    \begin{semilogxaxis}[clean,
        cycle list name=tol,
        xmode=log,
        log ticks with fixed point,
        xtick=data,
        table/col sep=comma,
        xmin = 23,
        xmax = 550,
        ymin = 1.9,
        ymax = 6.55,
        xlabel = {Ensemble Size (N)},
        ylabel = {Mean Spatio-temporal RMSE},
        every axis plot/.append style={line width=2pt, mark size=3.5pt},
        legend style={at={(0.67,0.85)},anchor=center},
        legend cell align={left}]
    \addplot table [x=N, y=EnGMF, col sep=comma] {data/lorenz63newresults.csv};
    \addlegendentry{EnGMF };
    \addplot[dashed,color=tolblue,line width=1.25pt,forget plot,dash pattern=on 4pt off 8pt]%
    table [x=N, y=EnGMFp3s, col sep=comma] {data/lorenz63newresults.csv};
    \addplot[dashed,color=tolblue,line width=1.25pt,forget plot,dash pattern=on 4pt off 8pt] table [x=N, y=EnGMFm3s, col sep=comma] {data/lorenz63newresults.csv};

    \addplot table [x=N, y=AEnGMF, col sep=comma] {data/lorenz63newresults.csv};
    \addlegendentry{AEnGMF };
    \addplot[dashed,color=tolred,line width=1.25pt,forget plot,dash pattern=on 4pt off 8pt]%
    table [x=N, y=AEnGMFp3s, col sep=comma] {data/lorenz63newresults.csv};
    \addplot[dashed,color=tolred,line width=1.25pt,forget plot,dash pattern=on 4pt off 8pt] table [x=N, y=AEnGMFm3s, col sep=comma] {data/lorenz63newresults.csv};

    \addplot table [x=N, y=ELEnGMF, col sep=comma] {data/lorenz63newresults.csv};
    \addlegendentry{ELEnGMF };
    \addplot[dashed,color=tolyellow,line width=1.25pt,forget plot,dash pattern=on 4pt off 8pt]%
    table [x=N, y=ELEnGMFp3s, col sep=comma] {data/lorenz63newresults.csv};
    \addplot[dashed,color=tolyellow,line width=1.25pt,forget plot,dash pattern=on 4pt off 8pt] table [x=N, y=ELEnGMFm3s, col sep=comma] {data/lorenz63newresults.csv};
    
    \addplot[color=black!35,dotted] table [x=N,y expr=2.2165, mark=none]{data/L63ELEnGMF.csv};
    \addlegendentry{SIR ($N\rightarrow\infty$)};
    \end{semilogxaxis}
    \end{tikzpicture}
    \caption{A quantitative look at the performance of the EnGMF with CKDE prior estimation, the AEnGMF with AKDE prior estimation, and the ELEnGMF with ELKDE prior estimation for the Lorenz '63 equations. The $x$-axis represents the ensemble size $N$, and the $y$-axis represents mean spatio-temporal RMSE~\cref{eq:spatio-temporal-RMSE} over $12$ separate Monte Carlo runs. The solid lines represent the mean of the error while the dashed lines represent three standard deviations of the error over the Monte Carlo simulations. The gray dotted line represents the theoretical minimum RMSE computed using a sequential importance resampling (SIR) filter with $N=25000$ ensemble members and optimal rejuvenation.}
    \label{fig:lorenz63-experiment}
\end{figure}
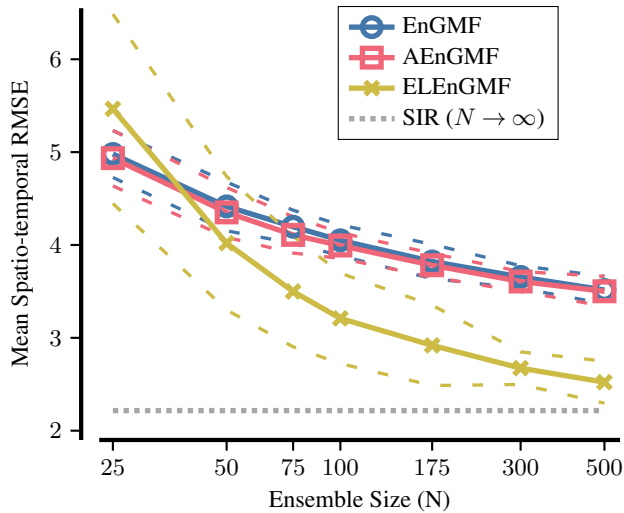

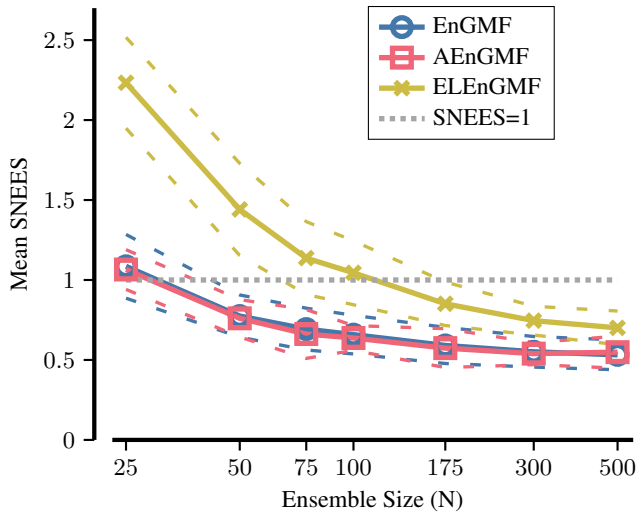
\begin{figure}[tp]
    \centering
    \begin{tikzpicture}
    \begin{semilogxaxis}[clean,
        cycle list name=tol,
        clip=true,
        xmode=log,
        log ticks with fixed point,
        xtick=data,
        table/col sep=comma,
        xmin = 23,
        xmax = 550,
        ymin = 0,
        ymax = 2.7,
        xlabel = {Ensemble Size (N)},
        ylabel = {Mean SNEES},
        every axis plot/.append style={line width=2pt, mark size=3.5pt},
        legend style={at={(0.67,0.85)},anchor=center},
        legend cell align={left}]
    \addplot table [x=N, y=EnGMF, col sep=comma] {data/lorenz63newresults_snees.csv};
    \addlegendentry{EnGMF };
    \addplot[dashed,color=tolblue,line width=1.25pt,forget plot,dash pattern=on 4pt off 8pt]%
    table [x=N, y=EnGMFp3s, col sep=comma] {data/lorenz63newresults_snees.csv};
    \addplot[dashed,color=tolblue,line width=1.25pt,forget plot,dash pattern=on 4pt off 8pt] table [x=N, y=EnGMFm3s, col sep=comma] {data/lorenz63newresults_snees.csv};

    \addplot table [x=N, y=AEnGMF, col sep=comma] {data/lorenz63newresults_snees.csv};
    \addlegendentry{AEnGMF };
    \addplot[dashed,color=tolred,line width=1.25pt,forget plot,dash pattern=on 4pt off 8pt]%
    table [x=N, y=AEnGMFp3s, col sep=comma] {data/lorenz63newresults_snees.csv};
    \addplot[dashed,color=tolred,line width=1.25pt,forget plot,dash pattern=on 4pt off 8pt] table [x=N, y=AEnGMFm3s, col sep=comma] {data/lorenz63newresults_snees.csv};

    \addplot table [x=N, y=ELEnGMF, col sep=comma] {data/lorenz63newresults_snees.csv};
    \addlegendentry{ELEnGMF };
    \addplot[dashed,color=tolyellow,line width=1.25pt,forget plot,dash pattern=on 4pt off 8pt]%
    table [x=N, y=ELEnGMFp3s, col sep=comma] {data/lorenz63newresults_snees.csv};
    \addplot[dashed,color=tolyellow,line width=1.25pt,forget plot,dash pattern=on 4pt off 8pt] table [x=N, y=ELEnGMFm3s, col sep=comma] {data/lorenz63newresults_snees.csv};
    
    \addplot[color=black!35,dotted] table [x=N,y expr=1, mark=none]{data/L63ELEnGMF.csv};
    \addlegendentry{SNEES=1};
    \end{semilogxaxis}
    \end{tikzpicture}
    \caption{A quantitative look at the performance of the EnGMF with CKDE prior estimation, the AEnGMF with AKDE prior estimation, and the ELEnGMF with ELKDE prior estimation for the Lorenz '63 equations. The $x$-axis represents the ensemble size $N$, and the $y$-axis represents the mean SNEES~\cref{eq:SNEES} over $12$ separate Monte Carlo runs. The solid lines represent the mean of the error while the dashed lines represent three standard deviations of the error over the Monte Carlo simulations. The gray dotted line represents the ideal SNEES of one. }
    \label{fig:lorenz63-snees}
\end{figure}

The final round of experiments aims to put the ELEnGMF to practical use in a sequential filtering experiment. In sequential filtering our aim is to find an optimal estimate of the truth $x^t$ at some time index $i$. This is performed by forecasting a state estimate to time index $i$ to obtain a prior and performing inference~\cref{eq:Bayesian-inference} to obtain a posterior estimate of the state $x^+_i$, then repeating the cycle over and over again \textit{ad infinitum}.

To this end we use the three variable Lorenz '63 equations~\citep{lorenz1963deterministic, strogatz2018nonlinear},
\begin{equation}
\begin{aligned}
    \dot{x}_1 &= 10 (x_2 - x_1),\\
    \dot{x}_2 &= x_1(28 - x_3) - x_2,\\
    \dot{x}_3 &= x_1 x_2 - \frac{8}{3}x_3,
\end{aligned}
\end{equation}
which is a chaotic test problem that is one of the foundational problems for particle filters~\citep{reich2015probabilistic}.

The following problem setup is taken from~\citep{popov2022adaptive,popov2023elengmf}. We perform inference every $\Delta t = 0.5$ time, with $5500$ sequential steps taken (the first $500$ of which are discarded to account for spinup---discarding errors from initial over- or under-confidence of the algorithm). 
We take the scalar range observation,
\begin{equation}
    h(x) = \left\lVert x - c_2\right\rVert_2,
\end{equation}
where $c_2 = \begin{bmatrix}6\sqrt{2} & 6\sqrt{2} & 27 \end{bmatrix}^T$ is the fixed point in the center of one of the wings of the Lorenz butterfly, with unbiased Gaussian observation error determined by the variance $R = 1$. All experiments are performed over 12 distinct Monte Carlo samples.

For the error metrics in this experiment, we look at the spatio-temporal root mean squared error (RMSE) and the scaled normalized estimation error squared (SNEES).
The spatio-temporal RMSE of the posterior means of the data with respect to the truth is,
\begin{equation}\label{eq:spatio-temporal-RMSE}
    \operatorname{RMSE}(x^+) = \sqrt{\frac{1}{\lvert x\rvert}\sum_{i,j} \left(x^+_{i,j} - x^t_{i,j}\right)^2}
\end{equation}
where $x^+$ is the temporal collection of posterior means, $x^t$ is the collection of true states of the system, $i$ is the time index, $j$ is the state index, and $\lvert x\rvert$ is the cardinality of the data. For the RMSE, the lower the better.
The SNEES~\citep{yun2022kernel} of the posterior means with respect to the truth is defined as,
\begin{equation}\label{eq:SNEES}
    \operatorname{SNEES}(x^+) = \frac{1}{n \lvert x \rvert}\sum_{i}(x^+_i - x^t_i)^T\left(\mathfrak{S}^+_{i,\text{global}}\right)^{-1}(x^+_i - x^t_i),
\end{equation}
where everything is the same as in~\cref{eq:spatio-temporal-RMSE}, except $\mathfrak{S}^+_{i,\text{global}}$, which represents the global covariance of the posterior uncertainty at time index $i$. A SNEES value of one is considered optimal, as the distribution of the error of the posterior with respect to the truth would have the same covariance as predicted by the filter. If the SNEES is less than one, then the filter is considered to be too conservative, while a filter with a SNEES greater than one is too confident in the estimate. In general, a conservative filter is preferable to one that is overconfident.

\begin{remark}
    The Lorenz '63 system has a Kaplan–Yorke dimension of $2.06$ (computed using a method based on~\citep{dieci2011numerical}) that is close to two, meaning that the covariance estimate in~\cref{eq:SNEES} can frequently become close to singular numerically. We thus discard outlier values (those that are greater than 100 under the sum) in order to ensure that these numerical effects don't bias our results in a manner that would make their utility hard to interpret.
\end{remark}

For this experiment we make use of the first projection method from~\cref{eq:naive-eigenvalue-bound}, with parameter $\epsilon_1 = 10^{-4}$ as this method produces covariance estimates that are similar to those obtained by performing the optimization procedure in~\cref{def:local-covariance}.

For the RMSE experiment in~\cref{fig:lorenz63-experiment}, it can be seen that the ELEnGMF performs worse than the CKDE or the AKDE in the case of $N=25$ ensemble members, but performs better in the case of $N \geq 100$ ensemble members, approaching the theoretical minimum RMSE (computed using a sequential importance resampling filter) at around $N=500$ ensemble members, having a significant improvement in error over the global covariance filters.

For the SNEES experiment in~\cref{fig:lorenz63-snees} it is evident that both the EnGMF and the AEnGMF behave almost identically, producing highly conservative estimates of the covariance of the system. This is not exactly surprising, as it can be shown~\citep{liu2016efficient} that the covariance of the prior distribution is a scaled by a factor of $1 + \beta^2_N$. The ELEnGMF is less conservative than the other two filters, as the individual estimates of the covariances allow the EnGMF to decrease the conservative nature of the prior estimate, translating to a less conservative estimate of the posterior.
While the trend of the SNEES for all three methods is to grow lower and lower relative to the ensemble size, as all the methods are equivalent to a bootstrap particle filter in the ensemble limit, the SNEES should increase after a certain point, though the amount of ensemble members needed to accomplish this is unknown, and impractical to compute.

\conclusions  
\label{sec:conclusions}

In this work we have introduced the E-localization methodology for kernel density estimation (ELKDE) and have applied it to the ensemble Gaussian mixture filter, creating the E-localized ensemble Gaussian mixture filter (ELEnGMF).
We have shown that the ELKDE is theoretically equivalent to canonical KDE methods in the Gaussian case, and is empirically more accurate on a highly non-Gaussian spiral distribution.
We have also shown that this methodology addresses one issue---the disparity between local and global notions of covariance---that arises in the application of KDE methods to distributions such as those arising from non-linear dynamics.
Additionally, through the application of the methodology to the Lorenz '63 equations, we have shown that the ELEnGMF has the potential to be a superior particle filter to the EnGMF in the sequential filtering setup.

While this work has tackled some issues inherent to the EnGMF, it has left many as open problems, and has even introduced some new issues.
In particular the selection of localization radius $r_i$ in~\cref{eq:choice-of-Si} is an issue of immediate interest. The methodology presented in this work for tackling this issue clearly has nice empirical behavior, but its justification is largely heuristic. Exploring this choice in future work is desirable.
Another, more significant issue is the choice of projection method in~\cref{eq:naive-eigenvalue-bound} and \cref{eq:smart-eigenvalue-bound}, and whether such methods are even necessary, or if a more robust alternative is viable.

Future non-theoretic work would involve applying adaptive covariance parameterization techniques~\citep{popov2022adaptive} to the ELEnGMF, such that the choice of bandwidth scaling factor~\cref{eq:bandwidth-scaling-factor} and radius scaling factor~\cref{eq:scaled-radius} could be performed adaptively in the sequential filtering regime.
Another future direction would involve applying the ELEnGMF to a practical orbit determination problem~\citep{yun2022kernel}.

\authorcontribution{AAP, EMZ, and RZ were responsible for developing the initial concept and research direction. AAP and EMZ were responsible for the methodology. AAP was responsible for the software implementation. AAP wrote the initial draft. AAP, EMZ, and RZ reviewed and edited the manuscript} 

\competinginterests{The authors declare that no competing interests are present.} 


\begin{acknowledgements}
This work was sponsored in part by DARPA (Defense Advanced Research Projects Agency) under STTR contract number W31P4Q-21-C-0032, and sponsored in part by the Air Force Office of Scientific Research (AFOSR) under award number: FA9550-22-1-0419.
\end{acknowledgements}







\bibliographystyle{copernicus}
\bibliography{bibfiles/covarianceshrinkage,bibfiles/em,bibfiles/engmf,bibfiles/filteringgeneral,bibfiles/kernelapproximation,bibfiles/misc,bibfiles/multifidelity,bibfiles/probability,bibfiles/problems,bibfiles/stochasticoptimization,bibfiles/banana}

\end{document}